\documentclass[12pt]{article}
\usepackage{amsmath,amssymb,amsbsy,amsfonts,amsthm,latexsym,
            amsopn,amstext,amsxtra,euscript,amscd}
\begin{document} 
\bibliographystyle{plain}

\newfont{\teneufm}{eufm10}
\newfont{\seveneufm}{eufm7}
\newfont{\fiveeufm}{eufm5}
%
%
\newfam\eufmfam
              \textfont\eufmfam=\teneufm \scriptfont\eufmfam=\seveneufm
              \scriptscriptfont\eufmfam=\fiveeufm
%
%
\def\frak#1{{\fam\eufmfam\relax#1}}
%


\def\bbbr{{\rm I\!R}} 
\def\bbbm{{\rm I\!M}}
\def\bbbn{{\rm I\!N}} 
\def\bbbf{{\rm I\!F}}
\def\bbbh{{\rm I\!H}}
\def\bbbk{{\rm I\!K}}
\def\bbbp{{\rm I\!P}}
\def\bbbone{{\mathchoice {\rm 1\mskip-4mu l} {\rm 1\mskip-4mu l}
{\rm 1\mskip-4.5mu l} {\rm 1\mskip-5mu l}}}
\def\bbbc{{\mathchoice {\setbox0=\hbox{$\displaystyle\rm C$}\hbox{\hbox
to0pt{\kern0.4\wd0\vrule height0.9\ht0\hss}\box0}}
{\setbox0=\hbox{$\textstyle\rm C$}\hbox{\hbox
to0pt{\kern0.4\wd0\vrule height0.9\ht0\hss}\box0}}
{\setbox0=\hbox{$\scriptstyle\rm C$}\hbox{\hbox
to0pt{\kern0.4\wd0\vrule height0.9\ht0\hss}\box0}}
{\setbox0=\hbox{$\scriptscriptstyle\rm C$}\hbox{\hbox
to0pt{\kern0.4\wd0\vrule height0.9\ht0\hss}\box0}}}}
\def\bbbq{{\mathchoice {\setbox0=\hbox{$\displaystyle\rm
Q$}\hbox{\raise
0.15\ht0\hbox to0pt{\kern0.4\wd0\vrule height0.8\ht0\hss}\box0}}
{\setbox0=\hbox{$\textstyle\rm Q$}\hbox{\raise
0.15\ht0\hbox to0pt{\kern0.4\wd0\vrule height0.8\ht0\hss}\box0}}
{\setbox0=\hbox{$\scriptstyle\rm Q$}\hbox{\raise
0.15\ht0\hbox to0pt{\kern0.4\wd0\vrule height0.7\ht0\hss}\box0}}
{\setbox0=\hbox{$\scriptscriptstyle\rm Q$}\hbox{\raise
0.15\ht0\hbox to0pt{\kern0.4\wd0\vrule height0.7\ht0\hss}\box0}}}}
\def\bbbt{{\mathchoice {\setbox0=\hbox{$\displaystyle\rm
T$}\hbox{\hbox to0pt{\kern0.3\wd0\vrule height0.9\ht0\hss}\box0}}
{\setbox0=\hbox{$\textstyle\rm T$}\hbox{\hbox
to0pt{\kern0.3\wd0\vrule height0.9\ht0\hss}\box0}}
{\setbox0=\hbox{$\scriptstyle\rm T$}\hbox{\hbox
to0pt{\kern0.3\wd0\vrule height0.9\ht0\hss}\box0}}
{\setbox0=\hbox{$\scriptscriptstyle\rm T$}\hbox{\hbox
to0pt{\kern0.3\wd0\vrule height0.9\ht0\hss}\box0}}}}
\def\bbbs{{\mathchoice
{\setbox0=\hbox{$\displaystyle     \rm S$}\hbox{\raise0.5\ht0\hbox
to0pt{\kern0.35\wd0\vrule height0.45\ht0\hss}\hbox
to0pt{\kern0.55\wd0\vrule height0.5\ht0\hss}\box0}}
{\setbox0=\hbox{$\textstyle        \rm S$}\hbox{\raise0.5\ht0\hbox
to0pt{\kern0.35\wd0\vrule height0.45\ht0\hss}\hbox
to0pt{\kern0.55\wd0\vrule height0.5\ht0\hss}\box0}}
{\setbox0=\hbox{$\scriptstyle      \rm S$}\hbox{\raise0.5\ht0\hbox
to0pt{\kern0.35\wd0\vrule height0.45\ht0\hss}\raise0.05\ht0\hbox
to0pt{\kern0.5\wd0\vrule height0.45\ht0\hss}\box0}}
{\setbox0=\hbox{$\scriptscriptstyle\rm S$}\hbox{\raise0.5\ht0\hbox
to0pt{\kern0.4\wd0\vrule height0.45\ht0\hss}\raise0.05\ht0\hbox
to0pt{\kern0.55\wd0\vrule height0.45\ht0\hss}\box0}}}}
\def\bbbz{{\mathchoice {\hbox{$\sf\textstyle Z\kern-0.4em Z$}}
{\hbox{$\sf\textstyle Z\kern-0.4em Z$}}
{\hbox{$\sf\scriptstyle Z\kern-0.3em Z$}}
{\hbox{$\sf\scriptscriptstyle Z\kern-0.2em Z$}}}}
\def\ts{\thinspace}

\newtheorem{theorem}{Theorem}
\newtheorem{lemma}[theorem]{Lemma}
\newtheorem{claim}[theorem]{Claim}
\newtheorem{cor}[theorem]{Corollary}
\newtheorem{prop}[theorem]{Proposition}
\newtheorem{definition}[theorem]{Definition}
\newtheorem{remark}[theorem]{Remark}
\newtheorem{question}[theorem]{Open Question}

\def\qed{\ifmmode
\squareforqed\else{\unskip\nobreak\hfil
\penalty50\hskip1em\null\nobreak\hfil\squareforqed
\parfillskip=0pt\finalhyphendemerits=0\endgraf}\fi}

\def\squareforqed{\hbox{\rlap{$\sqcap$}$\sqcup$}}

\def\cA{{\mathcal A}}
\def\cB{{\mathcal B}}
\def\cC{{\mathcal C}}
\def\cD{{\mathcal D}}
\def\cE{{\mathcal E}}
\def\cF{{\mathcal F}}
\def\cG{{\mathcal G}}
\def\cH{{\mathcal H}}
\def\cI{{\mathcal I}}
\def\cJ{{\mathcal J}}
\def\cK{{\mathcal K}}
\def\cL{{\mathcal L}}
\def\cM{{\mathcal M}}
\def\cN{{\mathcal N}}
\def\cO{{\mathcal O}}
\def\cP{{\mathcal P}}
\def\cQ{{\mathcal Q}}
\def\cR{{\mathcal R}}
\def\cS{{\mathcal S}}
\def\cT{{\mathcal T}}
\def\cU{{\mathcal U}}
\def\cV{{\mathcal V}}
\def\cW{{\mathcal W}}
\def\cX{{\mathcal X}}
\def\cY{{\mathcal Y}}
\def\cZ{{\mathcal Z}}
\newcommand{\rmod}[1]{\: \mbox{mod}\: #1}

\def\tcN{\cN^\mathbf{c}}

\def\Tr{{\mathrm{Tr}}}

\def\mand{\qquad \mbox{and} \qquad}
\renewcommand{\vec}[1]{\mathbf{#1}}

\def\eqref#1{(\ref{#1})}


\newcommand{\ignore}[1]{}

\hyphenation{re-pub-lished}

\parskip 1.5 mm
\def\lln{{\mathrm Lnln}}
\def\Res{\mathrm{Res}\,}

\def\F{{\bbbf}}
\def\Fp{\F_p}
\def\fp{\Fp^*}
\def\Fq{\F_q}
\def\ff{\F_2}
\def\ffn{\F_{2^n}}

\def\K{{\bbbk}}
\def \Z{{\bbbz}}
\def \N{{\bbbn}}
\def\Q{{\bbbq}}
\def \R{{\bbbr}}

\def\Zm{\Z_m}
\def \Um{{\mathcal U}_m}

\def \Bf{\frak B}

\def\Km{\cK_\mu}

\def\va {{\mathbf a}}
\def \vb {{\mathbf b}}
\def \vc {{\mathbf c}}
\def\vx{{\mathbf x}}
\def \vr {{\mathbf r}}
\def \vv {{\mathbf v}}
\def\vu{{\mathbf u}}
\def \vw{{\mathbf w}}
\def \vz {{\mathbfz}}

\def\\{\cr}
\def\({\left(}
\def\){\right)}
\def\fl#1{\left\lfloor#1\right\rfloor}
\def\rf#1{\left\lceil#1\right\rceil}

\def\flq#1{{\left\lfloor#1\right\rfloor}_q}
\def\flp#1{{\left\lfloor#1\right\rfloor}_p}
\def\flm#1{{\left\lfloor#1\right\rfloor}_m}

\def\Al{{\sl Alice}}
\def\Bob{{\sl Bob}}

\def\Or{{\mathcal O}}

\def\inv#1{\mbox{\rm{inv}}\,#1}
\def\invM#1{\mbox{\rm{inv}}_M\,#1}
\def\invp#1{\mbox{\rm{inv}}_p\,#1}

\def\Ln#1{\mbox{\rm{Ln}}\,#1}

\def \nd {\,|\hspace{-1.2mm}/\,}

\def\ord{\mu}

\def\E{\mathbf{E}}

\def\Cl{{\mathrm {Cl}}}

\def\epp{\mbox{\bf{e}}_{p-1}}
\def\ep{\mbox{\bf{e}}_p}
\def\eq{\mbox{\bf{e}}_q}

\def\bm{\bf{m}}

\newcommand{\floor}[1]{\lfloor {#1} \rfloor}

\newcommand{\comm}[1]{\marginpar{%
\vskip-\baselineskip 
\raggedright\footnotesize
\itshape\hrule\smallskip#1\par\smallskip\hrule}}

\def\rem{{\mathrm{\,rem\,}}}
\def\dist {{\mathrm{\,dist\,}}}
\def\etal{{\it et al.}}
\def\ie{{\it i.e. }}
\def\veps{{\varepsilon}}
\def\eps{{\eta}}

\def\ind#1{{\mathrm {ind}}\,#1}
               \def \MSB{{\mathrm{MSB}}}
\newcommand{\abs}[1]{\left| #1 \right|}

\title{On the  Sum-Product Problem 
on Elliptic Curves}

\author {{\sc Omran Ahmadi}  \\
{Department of Combinatorics \& Optimization}\\
{University of Waterloo} \\
{Waterloo, Ontario, N2L 3G1, Canada} \\
{\tt oahmadid@math.uwaterloo.ca} \and
{\sc Igor  Shparlinski}\\
             Department of Computing\\ Macquarie University\\
             Sydney, NSW 2109, Australia\\
             {\tt igor@comp.mq.edu.au} }
             
\date {}

\maketitle

\begin{abstract}
Let $\E$ be an ordinary elliptic curve over a finite field $\F_{q}$ of
$q$ elements and $x(Q)$ denote the $x$-coordinate of a 
 point $Q = (x(Q),y(Q))$ on $\E$. Given an $\F_q$-rational point $P$ of order $T$, 
we show that for any subsets $\cA, \cB$ of the unit group  
 of the residue ring modulo $T$, at least one of the sets
$$
\{x(aP) + x(bP)\ : \ a \in \cA,\  b \in \cB\} \quad\text{and}\quad
\{x(abP)\ : \ a \in \cA,\  b \in \cB\}
$$
is large. This question is motivated by a series of recent results on the 
sum-product problem over finite fields and other algebraic structures.  
\end{abstract}


\section{Introduction}

We fix   an ordinary elliptic curve $\E$ over  a finite field $\F_{q}$ of
$q$ elements. 

We assume that $\E$ is 
given by an affine Weierstra\ss\ equation
$$\E:\ y^2+(a_1x+a_3)y=x^3+a_2x^2+a_4x+a_6,$$
with some $a_1, \ldots, a_6 \in \F_{q}$, see~\cite{Silv}.

We recall that the set of all points on $\E$ forms an Abelian group, 
with the point at infinity $\cO$ as the neutral
element. 
As usual, we write every point $Q \ne \cO$ on $\E$ as $Q = (x(Q), y(Q))$. 

Let $\E(\F_{q})$   denote the set of $\F_q$-rational
points on  $\E$ and let $P \in \E(\F_{q})$ be a fixed point 
of order $T$. 

Let $\Z_T$  denote the residue ring modulo $T$ and 
let $\Z_T^*$ be its unit group.

We show that for any sets $\cA, \cB \subseteq \Z_T^*$, at least one of the sets
\begin{equation}
\label{eq:S and T}
\begin{split}
\cS  & = \{x(aP) + x(bP)\ : \ a \in \cA,\  b \in \cB\},\\
\cT & = \{x(abP)\ : \ a \in \cA,\  b \in \cB\}, 
\end{split} 
\end{equation}
is large.

This question is motivated by a series of recent results on the 
sum-product problem over $\F_q$ which assert that for 
any sets $\cA, \cB \subseteq \F_q$, at least one of the sets
$$
\cG = \{a+b\ : \ a \in \cA, \ b \in \cB\} \mand
\cH = \{ab\ : \ a \in \cA, \ b \in \cB\}
$$
is large, see~\cite{BGK,BKT,Gar1,Gar2,HaIoKoRu,HaIoSo,KatzShen1,
KatzShen2} for the background and further references.

We remark that yet another variant  of the 
sum-product problem for elliptic curves has recently 
been considered in~\cite{Shp2} where it is shown that 
for sets $\cR, \cS \subseteq \E(\F_q)$ at least one of the sets
$$
\{x(R) + x(S)\ : \ R \in \cR,\  S \in \cS\} \quad\text{and}\quad
\{x(R\oplus S)\ : \ R \in \cR, \ S \in \cS\}
$$
is large,  where  $\oplus $ denotes the group operation
on the points of $\E$. 

As in~\cite{Shp2}, our approach is based on
the argument of M.~Garaev~\cite{Gar2} which we combine 
with a bound of certain bilinear character sums
over points of $\E(\F_{q})$ which have been considered
in~\cite{BFGS} (instead of the
estimate of~\cite{Shp1} used in~\cite{Shp2}). 

In fact here we present a slight improvement of 
the result of~\cite{BFGS} that is based on 
using the argument of~\cite{GarKar}.

Throughout the paper, the implied constants in the symbols `$O$' and
`$\ll$' may  depend on an integer parameter
$\nu \ge 1$. We recall that $X\ll Y$
and $X = O(Y)$ are both equivalent to the inequality $|X|\le c Y$ with some
constant $c> 0$.

\bigskip

\noindent{\bf Acknowledgements.}
This paper was initiated during a very enjoyable visit of I.~S. at
the Department of Combinatorics \& Optimization of the University
of Waterloo whose hospitality, support and stimulating
research atmosphere are gratefully appreciated. Research of I.~S.
was supported by ARC grant DP0556431.

\section{Bilinear Sums over Elliptic Curves}

Let 
\begin{equation}
\label{eq:Sum T}
T_{\rho, \vartheta} (\psi,\cK, \cM) = \sum_{k \in \cK}\left|\sum_{m \in \cM} 
\rho(k)\vartheta(m) \psi(x(kmP))\right|,
\end{equation} 
where   $\cK,\cM \subseteq \Z_T^*$, 
$\rho(k)$ and $\vartheta(m)$ are arbitrary  
complex functions supported on $\cK$ and  $\cM$ with
$$
|\rho(k) | \le 1, \ k \in \cK, \mand  |\vartheta(m)| \le 1, \ m \in \cM, 
$$
and $\psi$ is a nontrivial additive character
of $\F_q$. 

These sums have been introduced and estimated in~\cite{BFGS}. 
Here we obtain a stronger result by using the approach to 
sums of this type given in~\cite{GarKar}.

\begin{theorem}
\label{thm:BilinSum}
Let $\E$ be an ordinary elliptic
curve defined over $\Fq$, and let $P \in \E(\Fq)$ be a point of
order $T$. Then, for any fixed integer $\nu\ge 1$, 
for all subsets  $\cK,\cM \subseteq \Z_T^*$ and  complex functions $\rho(k)$ and $\vartheta(m)$ supported on $\cK$ and  $\cM$ with
$$
|\rho(k) | \le 1, \ k \in \cK, \mand  |\vartheta(m)| \le 1, \ m \in \cM, 
$$
uniformly over all 
nontrivial additive characters $\psi$
of $\F_q$
$$
T_{\rho, \vartheta} (\psi,\cK, \cM)  \ll (\#\cK)^{1-\frac{1}{2\nu}} (\#\cM)^{\frac{\nu+1}{\nu+2}}T^{\frac{\nu+1}{\nu(\nu+2)}}q^{\frac{1}{4(\nu+2)}}(\log q)^{\frac{1}{\nu+2}}.
$$
\end{theorem}

\begin{proof} We follow the scheme of the proof of~\cite[Lemma~4]{GarKar}
in the special case of $d=1$
(and also $\Z_T$ plays the role of $\Z_{p-1}$). 
Furthermore, 
in our proof $\cK$, $\cM$, $\Z_T^*$ play the roles 
of $\cX$, $\cL_d$ and $\cU_d$ in 
the proof of~\cite[Lemma~4]{GarKar}, respectively. 
In particular, for some integer parameter $L$ with 
\begin{equation}
\label{eq:L small}
1 \le L \le T (\log q)^{-2}
\end{equation}
we define $\cV$ as  the 
set of the first $L$ prime numbers which do not divide $\#\E(\F_q)$
(clearly we can assume that, say $T \ge  (\log q)^{3}$, 
since otherwise the bound is trivial). We also note that in this case
\begin{equation}
\label{eq:v small 1}
\max_{v\in\cV} v = O(\#\cV \log q). 
\end{equation}

 Then 
we arrive to the following analogue of~\cite[Bound~(4)]{GarKar}:
$$
T_{\rho, \vartheta} (\psi,\cK, \cM) 
\le \frac{(\#\cK)^{1-1/(2\nu)}}{\#\cV} \sum_{t \in \Z_T^*} M_t^{1/(2\nu)}
$$
where
$$
M_t = \sum_{z \in \Z_T} \left| \sum_{v \in \cV} \vartheta(vt)\chi_\cM(vt)
\psi\(x(zvP)\) \right|^{2\nu}
$$
and $\chi_\cM$ is the characteristic function of the set $\cM$.
We only deviate from that 
proof at the point  where the Weil bound is applied to the sums
$$
\sum_{z\in \cH} \exp\(\frac{2 \pi i a}{p}\(\sum_{j=1}^{\nu} z^{tv_j}
- \sum_{j=\nu+1}^{2\nu} z^{tv_j}\)\)
\ll  \max_{1 \le j \le 2\nu } v_j q^{1/2}
$$ 
where $\cH$ is an arbitrary subgroup of 
$\F_q^*$ and  $v_1, \ldots, v_{2 \nu}$ are positive integers 
(such that $(v_{\nu +1}, \ldots, v_{2\nu})$ is not a permutation 
of $(v_1, \ldots, v_{\nu})$). Here,  as in~\cite{BFGS} 
we use instead the following bound from~\cite{LanShp}:
$$
\sum_{\substack{Q \in \cH\\ Q \ne \cO}}
\psi\(\sum_{j=1}^{\nu} x\(v_jQ\) - 
\sum_{j=\nu + 1}^{2\nu} x\( v_jQ\) \)  \ll  \max_{1 \le j \le 2\nu } v_j^2q^{1/2},
$$
where $\cH$ is the  subgroup of  $\E(\Fp)$ 
(in our particular case $\cH =\langle P\rangle$ is
generated by $P$) and $v_1, \ldots, v_{2 \nu}$ 
are the same as in the above, that is, such that 
$(v_{\nu +1}, \ldots, v_{2\nu})$ is not a permutation 
of $( v_1, \ldots,  v_{\nu})$. 

Now since $\#\E(\F_q)=O(q)$, using an argument similar 
to the one given in~\cite{GarKar} 
and recalling~\eqref{eq:v small 1} we obtain
\begin{eqnarray*}
M_t&\ll& \sum_{v_1\in \cV}\ldots\sum_{v_\nu\in \cV}\(\prod_{j=1}^\nu\chi_{\cM}(v_jt)\)T\\
& &\quad + 
\sum_{v_1\in \cV}\ldots\sum_{v_{2\nu}\in \cV}\(\prod_{j=1}^{2\nu}\chi_{\cM}(v_jt)\)q^{1/2}(\#\cV\log q)^2.
\end{eqnarray*}
Therefore
$$
M_t\ll \(\sum_{v\in \cV}\chi_{\cM}(vt)\)^\nu T+ 
\(\sum_{v\in \cV}\chi_{\cM}(vt)\)^{2\nu}q^{1/2}(\#\cV\log q)^2.
$$
This leads to the following
\begin{eqnarray*}
T_{\rho, \vartheta} (\psi,\cK, \cM) 
&\ll&\frac{(\#\cK)^{1-\frac{1}{2\nu}}}{\#\cV} T^\frac{1}{2\nu}\sum_{t \in \Z_T^*}\(\sum_{v\in \cV}\chi_{\cM}(vt)\)^{1/2}\\
& & \quad +~\frac{(\#\cK)^{1-\frac{1}{2\nu}}}{\#\cV}\(\#\cV\log q\)^{1/{\nu}} q^{1/{4\nu}}\sum_{t \in \Z_T^*}\(\sum_{v\in \cV}\chi_{\cM}(vt)\).
\end{eqnarray*}
On the other hand we have
$$
\sum_{t \in \Z_T^*}\(\sum_{v\in \cV}\chi_{\cM}(vt)\)=\#\cM\#\cV,
$$
and by the Cauchy  inequality we get
\begin{eqnarray*}
\sum_{t \in \Z_T^*}\(\sum_{v\in \cV}\chi_{\cM}(vt)\)^{1/2}&\le& (\#Z_T^*)^{1/2} \(\sum_{t \in \Z_T^*}\sum_{v\in \cV}\chi_{\cM}(vt)\)^{1/2}\\&\le& T^{1/2} \(\#\cM\#\cV\)^{1/2}.
\end{eqnarray*}
Thus
\begin{equation}
\label{eq:final}
\begin{split}
T_{\rho, \vartheta} (\psi,\cK, \cM) 
\ll&~\frac{(\#\cK)^{1-\frac{1}{2\nu}}}
{\(\#\cV\)^{1/2}} T^{1/{2\nu+1/2}}\(\#\cM\)^{1/2}\\
  & \quad +~(\#\cK)^{1-\frac{1}{2\nu}}\(\#\cV\log q\)^{1/{\nu}} q^{1/{4\nu}}\#\cM.
\end{split}
\end{equation}
Let 
$$
L =\fl{\frac{T^{\frac{v+1}{v+2}}}{q^{\frac{1}{2(v+2)}}{\(\log q\)}^{\frac{2}{v+2}}{\(\#\cM\)}^{\frac{v}{v+2}}}}.
$$
We note that if $L=0$ then
$$
T^{\frac{v+1}{v+2}}\le q^{\frac{1}{2(v+2)}}{\(\log q\)}^{\frac{2}{v+2}}{\(\#\cM\)}^{\frac{v}{v+2}}\le
q^{\frac{1}{2(v+2)}}{\(\log q\)}^{\frac{2}{v+2}}T^{\frac{v}{v+2}}
$$ 
and thus
$$
T \le q^{1/2} \(\log q\)^2.
$$
It is easy to check that in this case  
\begin{eqnarray*}
\lefteqn{\frac{ 
(\#\cK)^{1-\frac{1}{2\nu}} (\#\cM)^{\frac{\nu+1}{\nu+2}}T^{\frac{\nu+1}{\nu(\nu+2)}}q^{\frac{1}{4(\nu+2)}}(\log q)^{\frac{1}{\nu+2}} } { \#\cK  \#\cM} }\\
&  & \qquad \ge (\#\cK)^{-\frac{1}{2\nu}} (\#\cM)^{-\frac{1}{\nu+2}}T^{\frac{\nu+1}{\nu(\nu+2)}}q^{\frac{1}{4(\nu+2)}}(\log q)^{\frac{1}{\nu+2}}\\
&  & \qquad \ge T^{-\frac{1}{2\nu}} T^{-\frac{1}{\nu+2}}T^{\frac{\nu+1}{\nu(\nu+2)}}q^{\frac{1}{4(\nu+2)}}(\log q)^{\frac{1}{\nu+2}}\\
&  & \qquad = T^{-\frac{1}{2(\nu+2)}}q^{\frac{1}{4(\nu+2)}}(\log q)^{\frac{1}{\nu+2}} \ge 1.
\end{eqnarray*}
thus the result is trivial.

We now assume that $L \ge 1$ and choose $\cV$ to be of cardinality
$\#\cV= L$. Then we have 
$$
\frac{T^{\frac{v+1}{v+2}}}{q^{\frac{1}{2(v+2)}}{\(\log q\)}^{\frac{2}{v+2}}{\(\#\cM\)}^{\frac{v}{v+2}}}\ge
\#\cV\ge  \frac{T^{\frac{v+1}{v+2}}}{2q^{\frac{1}{2(v+2)}}{\(\log q\)}^{\frac{2}{v+2}}{\(\#\cM\)}^{\frac{v}{v+2}}}, 
$$
and $L\le T(\log q)^{-2}$ provided that $q$ is large enough. Now the result follows from~\eqref{eq:final}. 
\end{proof}

\section{Lower Bound for the Sum-Product Problem on Elliptic Curves}

\begin{theorem}
\label{thm:EC Sum Prod L} 
Let 
$\cA$ and  $\cB$ be arbitrary subsets of $\Z_T^*$.  
Then for the sets $\cS$ and $\cT$, given by~\eqref{eq:S and T},  we have
$$
\# \cS \#\cT \gg 
\min\{ q\# \cA,  (\# \cA)^{2} (\#\cB)^{5/3} q^{-1/6}
T^{-4/3}(\log q)^{-2/3}\} .
$$
\end{theorem}

\begin{proof} Let 
$$
\cH = \{ab\ : \ a \in \cA, \ b \in \cB\}.
$$
Following the idea of  M.~Garaev~\cite{Gar2}, 
we now denote by $J$ the number of solutions 
$(b_1, b_2 ,h,u)$ to  the
equation 
\begin{equation}
\label{eq:Eqn}
x(h b_1^{-1}P)  + x(b_2P) = u, \qquad  b_1,b_2 \in \cB, \ h \in
\cH,\ u \in \cS.
\end{equation} 
Since obviously the vectors
$$
(b_1, b_2 ,h,u) = \(b_1,b_2,a b_1, x(aP)+x(b_2P)\), \qquad a \in
\cA, \ b_1,b_2 \in \cB,
$$
are all pairwise distinct solutions to~\eqref{eq:Eqn}, we obtain
\begin{equation}
\label{eq:J Lower}
J \ge \# \cA (\# \cB)^2.
\end{equation} 

To obtain an upper bound on $J$ we use $\varPsi$ to denote the 
set of all $q$ additive characters of $\F_q$ and write $\varPsi^*$
for the set of nontrivial characters. 
Using the identity
\begin{equation}
\label{eq:Ident}
\frac{1}{q}\sum_{\psi \in \varPsi} \psi(z) =
\left\{\begin{array}{ll}
0,&\quad\text{if $z \in \F_q^*$,}\\
1,&\quad\text{if $z =0$,}
\end{array}
\right.
\end{equation}
we obtain 
\begin{eqnarray*}
J &=& \sum_{b_1\in \cB}\sum_{b_2\in \cB}
\sum_{h \in \cH} \sum_{u \in \cS} \frac{1}{q}\sum_{\psi \in \varPsi} \psi\(x(h b_1^{-1}P)  - x(b_2P) - u\)\\ &=& \frac{1}{q}\sum_{\psi \in \varPsi} \sum_{b_1\in \cB}
\sum_{h \in \cH}  \psi\(x(h b_1^{-1}P)\)
 \sum_{b_2\in \cB}\psi\(x(b_2P)\)\sum_{u \in \cS} 
\psi\( - u\)\\
&=&\frac{(\# \cB)^2\#\cS \#\cH}{q} \\
& &\qquad +~\frac{1}{q}\sum_{\psi \in \varPsi^*} \sum_{b_1\in \cB}
\sum_{h \in \cH}  \psi\(x(h b_1^{-1}P)\)
 \sum_{b_2\in \cB}\psi\(x(b_2P)\)\sum_{u \in \cS} 
\psi\( - u\).
\end{eqnarray*} 

Applying Theorem~\ref{thm:BilinSum} 
with $\rho(k) = \vartheta(m) = 1$,  $\cK = \cH$
and $\cM= \{b^{-1} ~:~ b \in \cB\}$ and also taking $\nu = 1$,   we obtain 
$$
\left| \sum_{b_1\in \cB}
\sum_{h \in \cH}  \psi\(x(h b_1^{-1}P)\)\right|
 \ll \Delta
$$
where
$$\Delta=(\#\cH)^{1/2} (\#\cB)^{2/3}T^{2/3}q^{1/12}(\log q)^{1/3}.
$$ 
Therefore, 
\begin{equation}
\label{eq:J prelim}
J\ll \frac{(\# \cB)^2\#\cS \#\cH}{q}
  +
\frac{1}{q} \Delta \sum_{\psi \in \varPsi^*}  \left|
 \sum_{b\in \cB}\psi\(x(bP)\)\right|\left|\sum_{u \in \cS} 
\psi\(-u\)\right|.
\end{equation}
Extending the summation over $\psi$ to the full set $\varPsi$ and using
the Cauchy inequality, we obtain 
\begin{equation}
\label{eq:Cauchy}
\begin{split}
\sum_{\psi \in \varPsi*}  \left|
 \sum_{b\in \cB}\psi\(x(bP)\)\right|&\left|\sum_{u \in \cS} 
\psi\(u\)\right| \\
\le &\sqrt{\sum_{\psi \in \varPsi}  \left|
 \sum_{b\in \cB}\psi\(x(bP)\)\right|^2}
\sqrt{\sum_{\psi \in \varPsi}  \left|\sum_{u \in \cS} 
\psi\(u\)\right|^2}.
\end{split} 
\end{equation}
Recalling the orthogonality property~\eqref{eq:Ident},
we derive
$$
\sum_{\psi \in \varPsi}  \left|
 \sum_{b\in \cB}\psi\(x(bP)\)\right|^2 = q\#\{(b_1, b_2) \in \cB^2\ :  \
 b_1 \equiv \pm  b_2\pmod T\} \ll  q \#\cB. 
$$
Notice that $b_1\equiv -b_2\pmod T$ has been included since $x(P)=x(-P)$ 
for $P \in \E(\F_q)$. 
 
Similarly,
$$
\sum_{\psi \in \varPsi}  \left|\sum_{u \in \cS} 
\psi\(u\)\right|^2 \le q \#\cS. $$
Substituting these bounds in~\eqref{eq:Cauchy} we obtain 
$$
\sum_{\psi \in \varPsi*}  \left|
 \sum_{b\in \cB}\psi\(x(bP)\)\right|\left|\sum_{u \in \cS}  
\psi\(u\)\right| \ll q\sqrt{\# \cB\#\cS},
$$
which after inserting in~\eqref{eq:J prelim}, yields
\begin{equation}
\label{eq:J Upper} 
J \ll \frac{(\# \cB)^2\#\cS \#\cH}{q}+ \Delta (\#\cS)^{1/2}(\#\cB)^{1/2} . 
\end{equation}
Thus, comparing~\eqref{eq:J Lower}
and~\eqref{eq:J Upper}, we derive
$$
 \frac{(\# \cB)^2\#\cS \#\cH}{q}+ \Delta (\#\cS)^{1/2}(\#\cB)^{1/2}\gg \# \cA (\# \cB)^2.
$$
Thus either 
\begin{equation}
\label{eq: case 1} 
\frac{(\# \cB)^2\#\cS \#\cH}{q} \gg \# \cA (\# \cB)^2,
\end{equation}
or 
\begin{equation}
\label{eq: case 2} 
 \Delta (\#\cS)^{1/2}(\#\cB)^{1/2}\gg \# \cA (\# \cB)^2.
\end{equation}
If~\eqref{eq: case 1} holds, then we have
$$
\#\cS \#\cH \gg q \# \cA. 
$$
If~\eqref{eq: case 2} holds, then
recalling the definition of $\Delta$, we derive
$$
(\#\cS)^{1/2}(\#\cH)^{1/2} (\#\cB)^{5/3}T^{2/3}q^{1/12}(\log q)^{1/3}
\gg\# \cA (\# \cB)^2. 
$$ 
It only remains to notice that $\#\cT \ge 0.5\#\cH$ 
to conclude the proof.
\end{proof}

We now consider several special cases. 

\begin{cor}  
\label{cor:1}
For any fixed $\varepsilon>0$ there exists $\delta > 0$ such that 
if $\cA, \cB \subseteq \Z_T^*$ are  arbitrary  subsets 
with 
$$
q^{1-\varepsilon} \ge \# \cA \ge \# \cB \ge T^{4/5+\varepsilon} q^{1/10},
$$
then for the sets $\cS$ and $\cT$, given by~\eqref{eq:S and T},  we have
$$
\# \cS \#\cT \gg \(\# \cA \)^{2 + \delta}. 
$$
\end{cor}

In particular, if $T\ge q^{1/2+\varepsilon}$ 
then there is always some  nontrivial range of cardinalities
$\# \cA$ and $\# \cB$
in which Corollary~\ref{cor:1} applies. 

\begin{cor}  
\label{cor:2}
for arbitrary  subsets of $\cA, \cB \subseteq \Z_T^*$  
with 
$$
 \# \cA = \# \cB \ge T^{1/2} q^{7/16} (\log q)^{1/4}
$$
and for the sets $\cS$ and $\cT$, given by~\eqref{eq:S and T},  we have
$$
\# \cS \#\cT \gg    q \# \cA .
$$
\end{cor}

\section{Upper Bound for the Sum-Product Problem on Elliptic Curves}

We now show that in some cases the sets $\cS$ and $\cT$
are not very big. 

As usual, we use $\varphi(T) = \# \Z_T^*$ to 
denote the Euler function.

\begin{theorem}
\label{thm:EC Sum Prod U} 
Let $q=p$ be prime and let $T \ge p^{3/4+\varepsilon}$.
Then there are sets 
$\cA = \cB \subset \Z_T^*$ of cardinality
$$
\# \cA = \# \cB =(1 + o(1)) \frac{\varphi(T)^2}{2p}
$$
such  that  for the sets $\cS$ and $\cT$, 
given by~\eqref{eq:S and T},  we have
$$
\max\{\# \cS, \#\cT \} \le (\sqrt{2}+o(1)) \sqrt{p\# \cA} 
$$
as $p\to \infty$. 
\end{theorem}

\begin{proof}  We recall the bound from~\cite{KohShp} of exponential sums over 
subgroups of the group of points on elliptic curves which in particular 
implies that for any subgroup $\cG$ of $\E(\F_p)$ the bound
\begin{equation}
\label{eq:Subgroup}
\sum_{G \in \cG}   \exp\(2 \pi i \lambda x(G)/p\) \ll p^{1/2},
\end{equation}
holds uniformly over all integer $\lambda$ with $\gcd(\lambda, p) =1$. 

Let $\mu(d)$ be the   M\"obius function, that is,  $\mu(1)
= 1$, $\mu(m) = 0$ if $m \ge 2$ is not square-free  and $\mu(m) = (-1)^{\omega(m)}$
otherwise, where $\omega(d)$ is the number of distinct prime divisors of $d\ge 2$,
see~\cite[Section~16.2]{HardyWright}. 
 
Using the inclusion-exclusion principle, we obtain
\begin{eqnarray*}
\sum_{\substack{a=1\\gcd(a,T) =1}}^T  \exp\(2 \pi i \lambda x(aP)/p\) 
& = & 
\sum_{d|T} \mu(d) \sum_{\substack{a=1\\d \mid a}}^T  
\exp\(2 \pi i \lambda x(aP)/p\)\\ 
& = & 
\sum_{d|T} \mu(d) \sum_{b=1}^{T/d}  
\exp\(2 \pi i \lambda x(bP)/p\).
\end{eqnarray*}

Using~\eqref{eq:Subgroup} and recalling that
$$
\sum_{d|T} 1  = T^{o(1)}
$$
see~\cite[Theorem~317]{HardyWright}, we derive
$$
\sum_{\substack{a=1\\gcd(a,T) =1}}^T 
 \exp\(2 \pi i \lambda x(aP)/p\) \ll p^{1/2 + o(1)}.
$$
Combining this with the  Erd{\H o}s-Tur\'{a}n inequality, 
see~\cite[Theorem~1.21]{DrTi},
 we see that for any positive integer $H$,
there are $H\varphi(T) /p  +  O\(p^{1/2+ o(1)}\)$ 
elements   $a \in \Z_T^*$ with  $x(aP) \in [0, H-1]$.
Let $\cA = \cB$ be the set of these elements $a$.
For the sets $\cS$ and $\cT$, we obviously have
$$
\# \cS \le 2H \mand \#\cT \le \varphi(T).
$$
We now choose $H = \varphi(T)/2$. Since $T \ge p^{3/4+\varepsilon}$
and  also since
$$
\varphi(T) \gg \frac{T}{\log \log T},
$$
see~\cite[Theorem~328]{HardyWright}, 
we have 
$$
\# \cA = \# \cB =\frac{\varphi(T)^2}{2p}  +  O\(p^{1/2+o(1)}\) =(1 + o(1)) \frac{\varphi(T)^2}{2p}
$$
as $p\to \infty$.
Therefore 
$$
\max\{\# \cS , \#\cT\} \le (\sqrt{2} + o(1)) \sqrt{p\#\cA}
$$
which concludes the proof. 
\end{proof}

We remark that if $T \ge p^{23/24+\varepsilon}$
, then the cardinality of the sets $\cA$ and $\cB$ 
of Theorem~\ref{thm:EC Sum Prod U}
is  
$$
\# \cA = \# \cB = T^{2 + o(1)} p^{-1} \ge 
T^{1/2} p^{7/16} (\log p)^{1/4}
$$
 and thus Corollary~\ref{cor:2} applies as well and we
have 
$$  (\sqrt{2} + o(1)) \sqrt{p\#\cA}
\ge \max\{\# \cS , \#\cT\} \ge \sqrt{\# \cS  \#\cT} \gg 
\sqrt{p\#\cA}
$$
showing that both  Corollary~\ref{cor:2} and
Theorem~\ref{thm:EC Sum Prod U} are tight in 
this range.

\section{Comments}

We remark that using Theorem~\ref{thm:BilinSum}
with other values of $\nu$ in the scheme of the 
proof of Theorem~\ref{thm:EC Sum Prod L} one can obtain 
a series of other statements. However they cannot be 
formulated as a lower bound on the product $\# \cS \#\cT$.
Rather they only give a lower bound on $\max\{\# \cS, \#\cT \}$
which however may in some cases be more precise than those which 
follow from Theorem~\ref{thm:EC Sum Prod L}. 

Certainly extending the range in which the upper and 
lower bounds on $\# \cS $ and $\#\cT$ coincide is also a very 
important question.


\begin{thebibliography}{99}

\bibitem{BFGS} W. D.~Banks, J. B. Friedlander,  M. Z. Garaev  and
I.~E.~Shparlinski,  `Double character sums over elliptic curves and
finite fields', 
{\it Pure and Appl. Math. Quart.\/}, {\bf 2} (2006),   179--197.

\bibitem{BGK} J. Bourgain,  A. A. Glibichuk and S. V. Konyagin,
`Estimates for the number of sums and products and for exponential
sums in fields of prime order', {\it J. Lond. Math. Soc.\/}, 
{\bf 73} (2006), 380--398.

\bibitem{BKT} J. Bourgain,  N. Katz and T. Tao,
`A sum product
estimate in finite fields and Applications',
{\it Geom. Funct. Analysis\/},  {\bf 14} (2004), 27--57.

\bibitem{DrTi} M. Drmota and  R. Tichy, {\it Sequences, discrepancies and
applications\/}, Springer-Verlag, Berlin, 1997.

\bibitem{Gar1}  M. Z. Garaev,  
`An explicit sum-product estimate in $\Fp$',  
{\it Intern. Math. Res. Notices\/},  
{\bf 2007} (2007), Article ID rnm035, 1--11.

 \bibitem{Gar2} M. Z. Garaev, `The sum-product estimate for large subsets of prime fields',  
{\it Preprint\/}, {\it Proc. Amerc. Math. Soc.\/},  (to appear).


 \bibitem{GarKar}  M. Z. Garaev and  A. A. Karatsuba, 
 `New estimates of double trigonometric sums with 
exponential functions',  
{\it Arch. Math.\/},  {\bf  87} (2006),   33--40.

\bibitem{HardyWright} G. H. Hardy and E. M. Wright, 
{\it An introduction to
the theory of numbers\/}, Oxford Univ. Press, Oxford, 1979.

\bibitem{HaIoKoRu} D. Hart, A. Iosevich, D. Koh and M. Rudnev,  
`Averages over hyperplanes, sum-product theory in
finite fields, and the Erdos-Falconer distance conjecture', 
{\it Preprint\/}, 2007 (available from {\tt
http://arxiv.org/abs/0707.3473}). 


 \bibitem{HaIoSo} D. Hart,  A. Iosevich and J. Solymosi,  
`Sums and products in finite fields via Kloosterman Sums', 
{\it Intern. Math. Res. Notices\/}, 
{\bf 2007} (2007), Article ID rnm007, 1--14.

 \bibitem{KatzShen1} N. H. Katz and C.-Y. Shen,  
`Garaev's inequality in finite fields not of prime order', 
{\it J.  Anal. Combin.\/},  {\bf 3} (2008), Article \#3.

 \bibitem{KatzShen2} N. H. Katz  and C.-Y. Shen,  
`A slight improvement to Garaev's sum product estimate', 
{\it Proc. Amerc. Math. Soc.\/},  (to appear).


\bibitem{KohShp} D. R. Kohel and I. E. Shparlinski, `Exponential sums and
group generators for  elliptic curves over finite fields',
{\it Proc. the 4th  Algorithmic Number Theory  Symp.,
Lect. Notes in
Comp. Sci.\/},
Springer-Verlag, Berlin, {\bf 1838} (2000),  395--404.

\bibitem{LanShp} T. Lange and I. E. Shparlinski,
`Certain exponential sums and random walks on elliptic curves',
{\it Canad. J. Math.\/}, {\bf 57}  (2005),   338--350.

 \bibitem{Shp1}
I.~E.~Shparlinski,  `Bilinear  character sums 
over elliptic curves', {\it Finite Fields and 
Their Appl.\/}, {\bf 14} (2008), 132--141.

 \bibitem{Shp2}
I.~E.~Shparlinski,  `On the elliptic curve analogue of the 
sum-product problem', {\it Finite Fields and 
Their Appl.\/}, (to appear).

\bibitem{Silv} J.~H.~Silverman, {\it The arithmetic of elliptic
curves\/},
Springer-Verlag, Berlin, 1995.


\end{thebibliography}
\end{document}